\documentclass[12pt]{amsart}
\usepackage{amsmath}
\usepackage{amsthm}
\usepackage{amsfonts}
\usepackage{amssymb}
\newtheorem{Theorem}{Theorem}[section]

\newtheorem{Lemma}[Theorem]{Lemma}

\theoremstyle{definition}

\theoremstyle{remark}

\numberwithin{equation}{section}

\newcommand{\Z}{{\mathbb Z}}
\newcommand{\R}{{\mathbb R}}
\newcommand{\C}{{\mathbb C}}

\newcommand{\SL}{{\mathcal {SL}}}

\begin{document}

\title[Toda flows]{Generalized Toda flows}

\author{Darren C.\ Ong}

\address{Department of Mathematics\\
Xiamen University Malaysia\\
Jalan Sunsuria\\
Bandar Sunsuria\\
43900 Sepang\\
Selangor Darul Ehsan, Malaysia}

\email{darrenong@xmu.edu.my}

\urladdr{www.xmu.edu.my/a/234.html}

\author{Christian Remling}

\address{Department of Mathematics\\
University of Oklahoma\\
Norman, OK 73019}

\email{christian.remling@ou.edu}

\urladdr{www.math.ou.edu/$\sim$cremling}

\date{January 3, 2018}



\begin{abstract}
The classical hierarchy of Toda flows can be thought of as an action of the
(abelian) group of polynomials on Jacobi matrices. We present a generalization of this
to the larger groups of $C^2$ and entire functions, and in this second case,
we also introduce associated cocycles and in fact give center stage to this object.
\end{abstract}
\maketitle
\section{Introduction}
A \textit{Jacobi matrix }is a difference operator of the form
\begin{equation}
\label{jac}
(Ju)_n = a_n u_{n+1} + a_{n-1}u_{n-1} + b_n u_n .
\end{equation}
Here, we assume that $a_n > 0$ and $b_n\in\mathbb R$ are bounded sequences, and then
$J$ is a bounded self-adjoint operator on $\ell^2(\Z)$. The space of all such Jacobi matrices
will be denoted by $\mathcal J$.
The alternative notation $\tau u$ for the difference expression from \eqref{jac} is employed when
we want to apply it to arbitrary sequences $u$, not necessarily from $\ell^2(\Z)$.

\textit{Toda flows }are global flows on $\mathcal J$, and one has one such flow for each polynomial $p$.
Please see, for example, \cite{Dickey,GeHol2,Teschl} for textbook style treatments and \cite{Dametal,Remtoda}
for recent work.
Or, as advertised in \cite{Remtoda}, we can think of the abelian group $G=\mathcal P = \R[x]$ of polynomials
(with pointwise addition as the group operation) acting on the space $\mathcal J$ of Jacobi matrices.

The Toda hierarchy is most conveniently constructed using the \textit{Lax equation}
\begin{equation}
\label{lax}
\dot{J} =  [p(J)_a, J] ;
\end{equation}
here, the anti-symmetric part
$p(J)_a$ of $p(J)$ is defined via its matrix representation in terms of the standard unit vectors $\delta_n\in\ell^2$. In other words,
if we write $X_{jk}=\langle \delta_j, X\delta_k\rangle$ for a bounded self-adjoint
operator $X$, then $(X_a)_{jk}=X_{jk}$ for $j<k$ and $=-X_{jk}$ if $j>k$,
and $(X_a)_{jj}=0$. The time derivative $\dot{J}$ of $J(t)$ is defined in the obvious way as the limit of the difference quotients
with respect to the operator norm.

The Lax equation \eqref{lax} defines a global flow \cite[Theorem 12.6]{Teschl},
and to define the group action $p\cdot J$ that was mentioned above, we take the time one map of this flow.
This \textit{is }an action of the abelian group $\mathcal P$ because, as is well known, any two Toda flows commute.
Notice also that the right-hand side of \eqref{lax} is linear in $p$, and, in particular, multiplying $p$ by a constant amounts
to the same as rescaling time. This means that in terms of the group action, the solution $J(t)$ to \eqref{lax} with the initial
value $J(0)=J$ is given by $J(t)=(tp)\cdot J$.

This group action has the following fundamental properties, most of which are classical and very well known:
the action $p\cdot J$ also commutes with the shift, which sends the coefficients of a Jacobi matrix to their shifted version
$(a_{n+1},b_{n+1})$. So we in fact have an action of the larger abelian group $G=\mathcal P\times\Z$, with $\Z$ acting
by shifts. The action is by unitary conjugation: $g\cdot J=U^*JU$ for some unitary operator $U=U(g,J)$. So all spectral
properties are preserved (the word \textit{isospectral }is often used in this context). In fact, the (generalized) reflection
coefficients are also preserved, and this is a property that has come into focus more recently \cite{Remtoda,Ryb}.

It is very natural to now wonder what would happen here if we consider more general functions $f$ instead of $p$,
and this is the subject of this paper. Of course, we can define $f(J)$ for general $f\in L^{\infty}$ via the spectral theorem,
and this will be a bounded operator. However, the potential problems start at the next step: the anti-symmetric (or upper triangular)
part of a bounded operator need not be bounded. This means that \eqref{lax} would need interpretation, and it also makes
it doubtful if this equation can then still define global flows and (everywhere defined) group actions, which is a very convenient
property of the Toda hierarchy, which we would like to keep.

We will approach things in two stages here. In the next section, we present a very simple recipe how we can
work around this issue for $f\in C^2$ and thus obtain global flows and an action of a large group from \eqref{lax}.
As expected, $f\cdot J$ can then also be obtained from what we have already by approximation, as
$f\cdot J=\lim p_n\cdot J$, for polynomials $p_n$ that converge to $f$ in a suitable sense. This in turn means
that many of the properties of the classical flows just carry over automatically.

There are important exceptions to this, though, and this is what we will discuss in the second part.
We do not know if $f\cdot J$ is unitarily equivalent to $J$ for a general $f\in C^2$; this property
will not just follow from approximation, at least not entirely. (It will follow from what we do in Section 2 that
$\sigma(f\cdot J)=\sigma(J)$, and the results of \cite{Remgenrefl} imply that the absolutely continuous parts of
multiplicity $2$ are unitarily equivalent.) One standard argument in the classical setting constructs
$U$ by solving $\dot{U}=P(t\cdot J)U$, but this becomes technically unpleasant if $P$ isn't bounded,
and the standard results from the literature \cite{Tana} don't seem to apply in any obvious way.

There is a second, different approach to these issues that was advertised in \cite{Remtoda}, and this
is what we will pursue here: if we have
an associated cocycle that extends the shift cocycle, then unitary equivalence and more will follow.
This approach, too, doesn't seem to work in complete generality. The natural class of functions $f$
for which the Toda cocycle can be generalized consists of
entire functions $f$, and this is what we will discuss in the second part of this note. Our choice of approach
is really very much a matter of taste since for entire functions, the first argument outlined above, via the equation
$\dot{U}=PU$, also works without too much trouble. We hope, however, that the detailed discussion
of cocycles in Sections 3, 4 and the appendix will also further illuminate various aspects of \cite{Remtoda}.

It appears that this second, more specialized scenario of the group
of entire functions acting on $\mathcal J$ via generalized Toda flows recommends itself most for
possible future use and investigation because of the additional machinery that is available here.
\section{The general Lax equation}
In this section, we assume that $f\in C^2(\mathbb R)$; of course, $f(J)$ only depends on the restriction of $f$
to the spectrum $\sigma(J)$ of $J$. We would like to consider the evolution equation $\dot{J}=X_f(J)$, with
$X_f(J)\equiv [f(J)_a, J]$. The existence of a second derivative will later ensure that $X(J)$ obeys a Lipschitz condition,
so that we have the usual Picard iteration method available.

But first we need to define $X_f(J)$; as pointed out in the introduction, $f(J)_a$ is not guaranteed to be a bounded
operator. We completely bypass this problem by simply interpreting $X(J)$ in terms of its matrix elements with respect to the
standard ONB $\{ \delta_n\}_{n\in\Z}$. In other words, we define
\[
X_f(J)_{jk} = \begin{cases} \left( [f(J)_a,J] \right)_{jk} & |j-k|\le 1 \\ 0 & \textrm{otherwise}
\end{cases} ;
\]
strictly speaking, this is somewhat formal, but it's clear how to interpret the formula rigorously:
in the commutator $f(J)_aJ-Jf(J)_a$, just multiply the infinite matrices in both terms according to
the usual row times column recipe. Since $J$ is tridiagonal, this is well defined. Then we take the $(j,k)$-element
of the resulting infinite matrix. Also, we have introduced the property of $X_f$ of being tridiagonal simply
by decree here; in the classical setting ($f=p$ a polynomial), this follows from the structure of $X$.

It is clear that $X_f(J)$ is formally symmetric, that is, $X_{jk}=X_{kj}$, and has bounded matrix elements,
so in fact defines a self-adjoint bounded tridiagonal operator on $\ell^2(\Z)$ (it's not quite a Jacobi matrix itself
because we don't know if the off-diagonal elements are positive).

We now refer to the theory of
\textit{operator differentiable functions }to control the dependence of $X_f(J)$ on $J$.
From \cite[Corollary 3.3]{Opdiff} we obtain that
\begin{equation}
\label{lip0}
\| f(J)-f(J')\| \le C \| J-J'\| , \quad \|J\|, \|J'\|\le R
\end{equation}
and here $C=2\left( \|f\|_{C[-R,R]}+\|f''\|_{L^2(-R,R)}\right)$ would work as the constant, and we have made no attempt
to optimize this; see \cite{Opdiff} for the subtleties of these questions.
Indeed, the precise form of the constant or of the conditions imposed on $f$ is not important for us,
what matters is that we now obtain the Lipschitz condition
\begin{equation}
\label{lip}
\| X_f(J)-X_f(J')\| \le L \|J-J'\|, \quad \|J\|, \|J'\| \le R ;
\end{equation}
again, we can take an $L$ here that only depends on $f$ and $R$.
\begin{Theorem}
\label{T2.1}
The Lax equation $\dot{J}=X_f(J)$ has a unique global solution for any initial value $J(0)=J$.
\end{Theorem}
So we have now defined Toda flows for more general functions $f$; as already discussed above,
our preferred viewpoint is to think of this as an action of the larger group $G=C^2(\R)$ on $\mathcal J$.
This of course assumes that any two flows commute, which we will show below.
As above, we then define $f\cdot J$ as $J(1)$, where $J(t)$ solves $\dot{J}=X_f(J)$, $J(0)=J$.
\begin{Theorem}
\label{T2.2}
If $f_n(x)\to f(x)$ uniformly on $|x|\le \|J\|$,
then $f_n\cdot J \to f\cdot J$ in operator norm.
\end{Theorem}
\begin{proof}[Proof of Theorems \ref{T2.1}, \ref{T2.2}]
For the most part, this will be a rather routine application of standard ODE techniques, except perhaps
for the existence of \textit{global }solutions, and here will we give an argument that intertwines
Theorems \ref{T2.1}, \ref{T2.2} in a curious way, so we prove them together.
Compare also \cite[Section 12.2]{Teschl} for the classical setting.

Write the Lax equation, with initial condition, as an integral equation (and here we can define
the operator valued integral via Riemann sums)
\begin{equation}
\label{de}
J(t) = J + \int_0^t X_f(J(s))\, ds ,
\end{equation}
and solve this by (Picard) iteration:
\[
J_0(t) = J, \quad J_{n+1}(t) = J + \int_0^t X_f(J_n(s))\, ds
\]
The Lipschitz condition \eqref{lip} guarantees that this converges locally (in $t$) to a solution since
the integral operator on the right-hand side of \eqref{de} is contractive; this property also gives uniqueness.
We obtain a unique solution on $0\le t\le T$, where $T>0$ only depends on the Lipschitz constant $L$
that we can achieve in a neighborhood of the initial value. Or, more specifically, we can say that $T>0$
can be chosen to depend on $f$ and (a bound on) $\|J\|$ only.

We have established uniqueness and \textit{local }existence of solutions, and we now turn to Theorem \ref{T2.2}.
Let's compare the solutions $J_g(t), J_h(t)$ for two different right-hand sides $X_g,X_h$, but with the same
initial value $J'$ at $t=0$. Throughout the following argument, we may have to restrict $t$ to a small interval
$0\le t\le T$, so that we can be sure that the solutions actually exist there.

Consider the Picard iterates
\[
J_0(t) = J_g(t), \quad J_{n+1}(t) = J' + \int_0^t X_h(J_n(s))\, ds
\]
for $J_h$, but with the added twist that we start the process not with the constant function equal to the initial value,
but with the comparison solution $J_g(t)$. Then, since $X_f$ is linear in $f$,
\[
J_1(t) - J_0(t) = \int_0^t X_{h-g}(J_0(s))\, ds .
\]
Now for a tridiagonal matrix, the operator norm is comparable to the $\ell^{\infty}$ norm of the sequence of matrix-elements;
more precisely, $\|J\|_{\infty}\le \|J\|_{\textrm{\rm op}}\le 3\|J\|_{\infty}$, where by $\|J\|_{\infty}$, we simply mean
$\max\{ \|a\|_{\infty}, \|b\|_{\infty} \}$, if we again denote the coefficients of $J$ by $a,b$.

Thus
\[
\|X_f(J)\|_{\infty} \le 4\|f(J)\| \|J\|_{\infty}\le 4\|f\|_{C[-\|J\|,\|J\|]} \|J\|_{\infty}
\]
(as above, we use the convention that
a norm of an operator with no subscript indicated refers to the operator norm),
and it follows that
\[
\|J_1(t)-J_0(t)\|_{\infty} \le 4Rt \|g-h\|_{C[-R,R]}, \quad R\equiv \sup_{0\le s\le T} \|J_0(s)\|_{\infty} .
\]
Note that $R$ depends only on $g$ and $J'$ (and in a moment we will see that in fact $R\le\|J'\|$).

With this preparation in place, we now run the usual Picard machine and can easily prove by induction that
\[
\|J_{n+1}(t)-J_n(t)\|_{\infty} \le \frac{4R\|g-h\|_{C[-R,R]}}{L_h} \frac{(L_ht)^{n+1}}{(n+1)!} ;
\]
here, $L_h$ denotes a Lipschitz constant for $X_h$, valid for $\|J\|\le 2\|J'\|$ (and we can make sure that
all $J_n(t)$ satisfy this bound by restricting our attention to a small enough interval $0\le t\le T$).
Since $J_n(t)\to J_h(t)$ in operator norm, uniformly on
$0\le t\le T$, and $J_g(t)=J_0(t)$, we see by summing these bounds that
\begin{equation}
\label{2.7}
\|J_g(t)-J_h(t)\|_{\infty} \le \frac{4R}{L_h} \|g-h\|_{C[-R,R]} (e^{L_ht}-1) .
\end{equation}
This now gives us the following variant of Theorem \ref{T2.2}: if $f_n\to f$ uniformly on $|x|\le 2\|J'\|$, say,
with uniformly bounded (in $L^2$) second derivatives there, then we can take $g=f$, $h=f_n$ in \eqref{2.7}, and
the Lipschitz constants $L_{f_n}$ will stay bounded, so $(tf_n)\cdot J'\to (tf)\cdot J'$, on an interval $0\le t\le T$,
which depends only on $\|J'\|$ and a uniform bound on the Lipschitz constants.

Now we return to Theorem \ref{T2.1}. Given an $f\in C^2(\R)$ and an initial value $J$, we approximate by
polynomials $p_n\to f$ on $[-2\|J\|,2\|J\|]$, with uniformly
bounded second derivatives. Then $(tp_n)\cdot J\to (tf)\cdot J$ for
$0\le t\le T$, and here we can take a $T>0$ that only depends on $f$ and $\|J\|$. In the classical case, for the polynomial flows,
we know that $(tp_n)\cdot J$ is unitarily equivalent to $J$; in particular, the operator norm is preserved. Thus also
$\|(tf)\cdot J\|=\|J\|$, and that means that in the first part of the proof,
we can also find a $T>0$ that works along the whole orbit: we solve on $0\le t\le T$, then take $(Tf)\cdot J$ as
the new initial value and can now be sure that we can also solve on $T\le t\le 2T$ etc.

Now that the existence of global solutions is guaranteed and with the extra information that
$\|(tf)\cdot J\|=\|J\|$, we can return to \eqref{2.7} and obtain the statement of Theorem \ref{T2.2}
from this. To do this, we let $g,h$ swap roles, that is, we take $g=f_n$, $h=f$, and then we don't need uniform control
on the Lipschitz constants. This works because we now know that $R\le\|J'\|$ satisfies a uniform (in $n$) bound.
\end{proof}

As an immediate consequence of Theorem \ref{T2.2}, we obtain that (generalized) Toda flows commute
with each other and the shift. In particular, it is indeed consistent to speak of an action of the abelian group
$G=C^2(\R)\times\Z$ on $\mathcal J$.
\begin{Theorem}
\label{T2.3}
The action itself $J\mapsto f\cdot J$ is continuous with respect to the operator norm, and in fact
\[
\|f\cdot J - f\cdot J'\|_{\infty} \le \|J-J'\|_{\infty} e^{L_f} .
\]
\end{Theorem}
This is \textit{proved }by the same methods: we run a Picard iteration that we start with
$J_0(t)=(tf)\cdot J - J + J'$, and we make $J'$ the initial value in the iteration, so $J_n(t)\to (tf)\cdot J'$.
We then use the same estimates as in the proof of Theorem \ref{T2.2}. We leave the details to the reader.
\section{Cocycles and the zero curvature equation}
In this section, we review and expand the perspective on Toda flows that was proposed in \cite{Remtoda}.
As in that reference, we introduce the notation
\[
\SL = \{ T: \C \to \textrm{\rm SL}(2,\C): T \textrm{ entire, } T(x)\in \textrm{\rm SL}(2,\R) \textrm{ for }
x\in\R \}
\]
for the group of matrix functions that our cocycles will take values in. Recall that given a group action $G\times X\to X$
on a space $X$, an \textit{$\SL$-cocycle }was defined as a function $T: G\times X\to\SL$ satisfying the
\textit{cocycle identity}
\begin{equation}
\label{ccid}
T(gh; x) = T(g; h\cdot x) T(h;x) .
\end{equation}
In our context, the most basic example is given by the shift cocycle: here, $G=\Z$ acts on $\mathcal J$ by shifts, and the cocycle
is given by the transfer matrix
\[
T(1;J) = A(J)\equiv \begin{pmatrix} \frac{z-b_1}{a_1} & \frac{1}{a_1} \\ -a_1 & 0 \end{pmatrix} ,
\]
and then $T(n;J)=A((n-1)\cdot J) \cdots A(J)$ for $n\ge 0$ and $T(n;J)=T(-n; n\cdot J)^{-1}$ for $n<0$.
These latter definitions are forced on us by the cocycle identity, but of course this is also how one would
normally have defined the transfer matrix if one wants to obtain a matrix that updates solution vectors
$Y(n)=(y_{n+1}, -a_ny_n)^t$, $\tau y=zy$, in the sense that $T(n)Y(0)=Y(n)$. This property of $T$ also makes it clear that the
shift cocycle updates the \textit{Titchmarsh-Weyl $m$ functions}
\[
m_{\pm}(z) = \mp \frac{f_{\pm}(1,z)}{a_0 f_{\pm}(0,z)} , \quad \tau f_{\pm} = z f_{\pm}, \quad
f_{\pm}\in\ell^2(\Z_{\pm})
\]
along the action:
\begin{equation}
\label{3.1}
\pm m_{\pm}(n\cdot J) = T(n;J)(\pm m_{\pm}(J)) ,
\end{equation}
and here an invertible matrix $M=\left( \begin{smallmatrix} a & b \\ c & d \end{smallmatrix}\right)$ acts on the
Riemann sphere $\C_{\infty}$ as a linear fractional transformation, $Mw = \frac{aw+b}{cw+d}$.

As was emphasized in \cite{Remtoda}, a crucial property of the classical Toda hierarchy is the existence of an
$\SL$-cocycle with the same basic properties. For each polynomial flow, there is a matrix function
$B(J)$, taking values in traceless entire matrix functions and again real on the real line, such that if $T(t;J)$
is defined as the solution of
\begin{equation}
\label{todacc}
\dot{T} = B(t\cdot J) T, \quad T(0)=1,
\end{equation}
then $T$ is an $\SL$-cocycle for the action of $G=\R$ on $\mathcal J$ by the corresponding Toda flow
$(tp)\cdot J$. In fact, much more is true: one obtains, in this way, a cocycle for the action of the much larger
group $G=\mathcal P\times\Z$, with again $\Z$ acting by shifts \cite[Theorem 2.2]{Remtoda}, and, as in \eqref{3.1}, this cocycle updates
the $m$ functions along the action:
\[
\pm m_{\pm}(g\cdot J) = T(g;J)(\pm m_{\pm}(J))
\]

The key here is the existence of a \textit{joint }cocycle that extends the shift cocycle. Such a joint cocycle
(even if the acting group is just $G=\R\times\Z$, corresponding to one flow plus the shift) will automatically
update the $m$ functions correctly \cite[Theorem 2.3]{Remtoda}.

Now we can ask ourselves how such joint cocycles, for an action of $G=\R\times\Z$, can arise, and the obvious
way to produce them would be to combine two individual cocycles.
The (differentiable) cocycles for an action of $G=\R$ are exactly given by \eqref{todacc}, in the following sense:
for any choice of a $B$ as above, $T$ will be an $\SL$-cocycle, and, conversely, if an $\SL$-cocycle is given, then
it will satisfy \eqref{todacc}, with $B(J)=(d/dt)T(t;J)\bigr|_{t=0}$. Now the key question is: when does such a cocycle
form a joint cocycle for the action of $G=\R\times\Z$, when combined with the shift cocycle?

This question has a simple answer: exactly when the \textit{zero curvature equation}
\begin{equation}
\label{zc}
\dot{A}(J) = B(1\cdot J)A(J) - A(J)B(J) , \quad \dot{A}\equiv \frac{d}{dt} A(t\cdot J) \bigr|_{t=0} ,
\end{equation}
holds, and this gives (we believe) a rather transparent interpretation of this equation.

This was also pointed out in \cite{Remtoda}, but in somewhat informal language,
so let us make it completely explicit here. Suppose that, as above, $B=B(J)$ is a continuous matrix function taking values in
\[
\mathfrak{sl} =\{ B: \C \to \C^{2\times 2}: B \textrm{ entire, }\textrm{\rm tr}\:B=0, B(x)\in\R^{2\times 2}
\textrm{ for }x\in\R \} ;
\]
the continuity requirement
refers to the operator norm on $\mathcal J$ and the topology of locally uniform convergence on $\mathfrak{sl}$.
Fix one of the flows from the general Toda hierarchy; in other words, fix an $f\in C^2(\R)$; it will then be convenient
to use the short-hand notation $t\cdot J\equiv (tf)\cdot J$ for this flow. Then $B$ yields an
$\SL$-cocycle for this action via \eqref{todacc}. In addition to this, we have the shift cocycle for the action of $\Z$.
We are now trying to glue these together to produce a joint cocycle for the action of $G=\R\times\Z$.
If this works at all, then, by the (anticipated) cocycle identity, the attempt
\begin{equation}
\label{3.2}
T(g;J) := T(t;n\cdot J) T(n; J) , \quad g=(t,n)\in G=\R\times\Z ,
\end{equation}
is as good as any. Note that on the right-hand side, we are only using the individual cocycles that we already
constructed, so of course this definition is not circular.
\begin{Theorem}
\label{Tzc}
\eqref{3.2} defines a cocycle for the action of $G=\R\times\Z$ if and only if the zero curvature equation
\eqref{zc} holds.
\end{Theorem}
\begin{proof}
First of all, we claim that we have a cocycle if and only if the identity
\begin{equation}
\label{3.3}
T(t;1\cdot J)A(J) = A(t\cdot J)T(t;J)
\end{equation}
holds for all $t\in\R$. Clearly, \eqref{3.3} is necessary: it is an immediate consequence of the cocycle identities
for $g=(t,1)=(t,0)(0,1)=(0,1)(t,0)$. Conversely, the cocycle identity \eqref{ccid}
for $g=(s,m)$, $h=(t,n)$, and $J$ for the $T$ defined by
\eqref{3.2} is easily seen to be equivalent to
\begin{equation}
\label{3.8}
T(t; m\cdot J')T(m;J')=T(m;t\cdot J')T(t; J') ,
\end{equation}
with $J'=n\cdot J$. Now the $m=1$ case of this is \eqref{3.3}, and then \eqref{3.8} in general follows by a
straightforward induction on $|m|$.

The zero curvature equation \eqref{zc} follows at once from \eqref{3.3} by taking the $t$ derivative at $t=0$ on both sides
and using \eqref{todacc}.
Conversely, assume now that \eqref{zc} holds. Call the two sides of \eqref{3.3} $L(t)$ and $R(t)$, respectively, and notice that
$dL/dt = B(t\cdot 1\cdot J) L$ and, by \eqref{todacc}, \eqref{zc}, and since $(d/dt)A(t\cdot J)=(d/ds)A(s\cdot t\cdot J)\bigr|_{s=0}$,
\[
\frac{dR}{dt} = \dot{A}(t\cdot J) T(t;J) + A(t\cdot J)B(t\cdot J) T(t;J) = B(1\cdot t\cdot J) R .
\]
So $L$ and $R$ solve the same ODE, and since also $L(0)=R(0)(=A(J))$, we obtain that $L=R$, as desired.
\end{proof}
\section{Cocycles for generalized Toda flows}
The discussion of the previous section suggests the following two step procedure: (1) extend the matrix
functions $B=B_p$ of the classical Toda hierarchy to (parts of) the general hierarchy; (2) show that these
$B$ satisfy the zero curvature equation.

This we will do for the group of functions
\[
\mathcal{O} = \{ f:\C\to\C : f \textrm{ entire, }f(x)\in\R \textrm{ for }x\in\R \} .
\]
For a given $J$, everything would in fact work the same way if $f$ is only holomorphic on $|z|<R$ for some $R>\|J\|$,
though we'd then have to work with cocycles which are also defined only on this disk $|z|<R$. In particular, the conclusions
of Theorem \ref{T3.4}, that $f\cdot J$ is unitarily equivalent to $J$ and the absolute values of the reflection coefficients
are preserved, are valid in this setting also. We prefer to have $\SL$-cocycles available and thus do not bother with
this generalization here.

We will obtain a cocycle for the action of $G=\mathcal{O}\times\Z$ and then all the benefits that come with it.
Once $B$ has been defined,
it will again be more convenient technically to obtain the desired results from the classical case by approximation,
so this is what we'll do here. In the appendix, we show how to derive the zero curvature equation from the Lax equation
for the classical (polynomial) flows, to make our treatment more self-contained and have a detailed proof written up of
this crucial step.

To define $B=B_f$ for $f\in\mathcal{O}$, we first of all introduce the (multiplication) operators (or, equivalently, sequences)
\[
g(z;J) = \left((J-z)^{-1}\right)_d , \quad h(z;J) = \left( 2aS(J-z)^{-1}\right)_d-1;
\]
here, $S$ again denotes the shift operator,
and $X_d$ refers to the diagonal part of a (let's say: bounded) operator $X$, again thought of as an infinite matrix with
respect to the standard basis. In other words, $X_d$ acts by multiplication by the sequence $X_{nn}=\langle \delta_n, X\delta_n\rangle$.

Note that $g,h$ are holomorphic on $|z|>\|J\|$, including $z=\infty$, and $g(\infty)=0$, $h(\infty)=-1$.
Thus we have the expansions
\begin{equation}
\label{3.6}
k(z) = \sum_{n\ge 0} k_nz^{-n}, \quad |z|>\|J\|,
\end{equation}
for $k=g,h$.

For a Laurent series $L(z)$ (about $z_0=0$), we denote its power series part by $[L]=\sum_{n\ge 0}L_nz^n$.
Using this notation, we can now define $B=B_f(J)\in\mathfrak{sl}$ for $f\in\mathcal{O}$, as follows:
\begin{equation}
\label{defB}
B = \begin{pmatrix} \left( [fh]+f(J)_d - 2(z-b)[fg]\right)_1 & -2\left( [fg]\right)_1 \\ 2\left(a^2[fg]\right)_0 &
\left( 2(z-b)[fg] -[fh]-f(J)_d \right)_1 \end{pmatrix}
\end{equation}
Here, the indices refer to the $n$ variable of the various multiplication operators (or sequences), so, for example,
$\left( [fg]\right)_1 = [fg_1]$, with $g_1=\langle \delta_1, (J-z)^{-1}\delta_1\rangle$.

Let's check that $B$ indeed takes values in $\mathfrak{sl}$. First of all, $B=B(z)$ is entire since $fg$, $fh$ are holomorphic
on $|z|>\|J\|$, so the power series parts $[fg]$, $[fh]$ converge everywhere. Moreover, using the Neumann series
\begin{equation}
\label{neum}
(J-z)^{-1} = -\sum_{n\ge 0} z^{-n-1}J^n , \quad |z|>\|J\|
\end{equation}
we see that the expansion coefficients of $g$ and $h$ are real, and so
are those of $f$, by assumption, so $B$ is real on the real line.

Next, comparison with formula (5.1) from \cite{Remtoda} and the discussion that follows (see also \cite[Section 12.4]{Teschl})
shows that if $f=p$ is a polynomial, then \eqref{defB} recovers the $B$ from the classical Toda cocycle. This of course is the
property that motivated our definition in the first place: \eqref{defB} is a natural extension of these formulae.

Finally, we observe that $B=B_f(J)$ is continuous in $f$ and $J$, in the following sense. Introduce the metric
\[
d(J,J') = \sum_{n\in\Z} 2^{-|n|} \left( |a_n-a'_n|+|b_n-b'_n| \right)
\]
on $\mathcal J$. Note that if $J_n\to J$ in operator norm, then $d(J_n,J)\to 0$ and also $\|J_n\|\le C$, but of course
these latter conditions are much weaker. Suppose now that $f_n(z)\to f(z)$ locally uniformly and $d(J_n,J)\to 0$, $\|J_n\|\le C$.
Then $B_n(z)\to B(z)$ locally uniformly,
where we have used the obvious notations $B_n\equiv B_{f_n}(J_n)$, $B=B_f(J)$.

To see this, observe first of all that \eqref{neum} gives us uniform bounds
\[
|g_k|, |h_k|\le C^k
\]
on the coefficients of $g,h$, which are valid for all $\|J\|\le C$.
We also have uniform (in $n$) bounds on the Taylor coefficients of the $f_n$, from Cauchy's estimates,
since, by assumption, $\sup \max_{|z|=R} |f_n(z)|<\infty$ for all $R>0$. The coefficients of $[f_ng]$, $[f_nh]$
are obtained as Cauchy products from those coefficients, so we have uniform bounds on these as well. This means
that to establish the locally uniform convergence $B_n(z)\to B(z)$, it is enough to verify
that the Taylor coefficients (about $z_0=0$) converge. In this form, the claim
is clear since the Taylor coefficients of $f_n(z)$ converge by assumption and those of $g(z;J_n)$, $h(z;J_n)$ are also easily seen to
converge, by again using the Neumann series. Moreover, the coefficient sequences $a,b$ are of course also continuous
functions of $J$.

So we now have a cocycle for each general Toda flow, generated by an $f\in\mathcal{O}$. We denote by
$T(f;J)$ the solution to \eqref{todacc}, evaluated at $t=1$.
The continuity of $B=B_f$ has the following important consequence.
\begin{Lemma}
\label{L3.1}
If $f_n,f\in\mathcal{O}$ and $f_n(z)\to f(z)$ locally uniformly, then $T(f_n;J)\to T(f;J)$ locally uniformly.
\end{Lemma}
\begin{proof}
From Theorem \ref{T2.2} and its proof, we know that $(tf_n)\cdot J\to (tf)\cdot J$, uniformly on $0\le t\le 1$,
and in operator norm (convergence in $d$ already would have been enough here). By the continuity properties
of $B$ that were just observed this gives us that $B_{f_n}((tf_n)\cdot J)\to B_f((tf)\cdot J)$, uniformly in $0\le t\le 1$
and locally uniformly in $z\in\C$. By standard ODE theory, the claim now follows.
\end{proof}
Now everything else falls into place more or less automatically. We again incorporate the shift
also, so consider the group $G=\mathcal{O}\times\Z$ and its action, and then also its cocycle $T(g;J)$, defined
as in \eqref{3.2} from the individual cocycles that we already have.
\begin{Theorem}
\label{T3.3}
$T(g;J)$ is a cocycle for the action of $G=\mathcal{O}\times\Z$. In particular, the zero curvature equation \eqref{zc}
holds for any $f\in\mathcal{O}$.
\end{Theorem}
\begin{proof}
To verify the cocycle identity \eqref{ccid}, we approximate the functions $f,k\in\mathcal{O}$ in
$g=(f,m)$ and $h=(k,n)$ by polynomials, and then we're back in the classical case and can refer
to \cite[Theorem 2.2]{Remtoda} and then apply Lemma \ref{L3.1} and a similar continuity property
of $T(f;J)$ with respect to the second argument, which is proved in the same way.

The zero curvature equation now follows automatically, from Theorem \ref{Tzc}. Alternatively, it
could have been derived directly by a similar approximation argument, after passing to the integrated
form of \eqref{zc}.
\end{proof}
\begin{Theorem}
\label{T3.4}
The cocycle from Theorem \ref{T3.3} updates $m_{\pm}$ correctly:
\begin{equation}
\label{4.1}
\pm m_{\pm}(g\cdot J) = T(g;J)(\pm m_{\pm}(J))
\end{equation}
for all $g\in G=\mathcal{O}\times\Z$. As a consequence, $g\cdot J$ is unitarily equivalent to $J$,
the absolute values of the generalized reflection coefficients are preserved,
and $g\cdot J$ and $J$ are reflectionless on the same sets.
\end{Theorem}
The \textit{generalized reflection coefficients }(which are defined for arbitrary Jacobi matrices, not necessarily
of classical scattering type) are discussed in detail in \cite{Remgenrefl}.
\begin{proof}
\eqref{4.1} follows from the same approximation argument, since if $J_n\to J$ (or just
$d(J_n,J)\to 0$), then, as is well known, $m_{\pm}(J_n)\to m_{\pm}(J)$ locally uniformly.
Or we could combine Theorem \ref{T3.3} with
\cite[Theorem 2.3]{Remtoda}. The remaining statements then follow from \cite[Theorem 2.4]{Remtoda}.
\end{proof}
\appendix
\section{Zero curvature from Lax}
In this appendix, we give a detailed derivation of the zero curvature equation \eqref{zc} for the classical (polynomial)
Toda flows, defined via their Lax equations \eqref{lax}. Since this step is of central importance in our scheme
(for example, this seems to be the
structurally most satisfying way of proving that the Toda cocycles update the $m$ functions),
there seems to be some point in giving a completely explicit treatment. The discussions of the Toda hierarchy that we have come across
either just postulate the zero curvature equation as an alternative starting point
(and substitute for the Lax equation), or they don't clearly draw the conclusion we want, even if they present
the relevant computations. Our treatment
will stay rather close to that of \cite[Section 12.2]{Teschl} (and see also \cite[Section 1.2]{GeHol2}),
with quite a few details filled in (and typographical errors in the key formula for $\dot{b}$ corrected).

So we want to establish:
\begin{Theorem}
\label{TA.1}
For a polynomial $f=p$, define $B$ by \eqref{defB}. Then $B$ obeys the zero curvature equation \eqref{zc} along the
flow defined by $f$.
\end{Theorem}
We begin with a calculation that will be used later in the proof.
\begin{Lemma}
\label{LA.1}
\[
\left( J^n\right)_+ = \sum_{k=0}^{n-1} a\left( (J^k)_dS-(SJ^k)_d\right)J^{n-k-1}
\]
\end{Lemma}
Here, $X_+$ denotes the upper triangular part of an operator (and not including the diagonal), as usual thought of as a matrix.
In particular, and this its significance here, we can then write the anti-symmetric part of $X$ as
$X_a=2X_+-X+X_d$.
\begin{proof}
We establish this by induction on $n$. We can write
\begin{equation}
\label{jacop}
J=aS+S^*a+b .
\end{equation}
Then $J_+=aS$, and since $S_d=0$, the $n=1$ case is now clear.

Now assume the identity of the Lemma holds for $n$. We want to show that it then holds for $n+1$
also, and with the help of the induction hypothesis, this statement may be rewritten as
\begin{equation}
\label{a1}
(J^{n+1})_+=\left( J^n\right)_+J + (J^n)_d J_+ - (J_+ J^n)_d  .
\end{equation}
Now
\begin{align*}
(J^{n+1})_+ & =((J^n)_+J)_+ + ((J^n)_d J)_+ \\
& = (J^n)_+J-((J^n)_+J)_d + ((J^n)_d J)_+
\end{align*}
and also $((J^n)_dJ)_+=(J^n)_dJ_+$, so \eqref{a1} becomes
$((J^n)_+J)_d = (J_+ J^n)_d$, and here
\[
((J^n)_+J)_d = ((J^n)_+ S^*a)_d = (J^n S^*a)_d .
\]
This gives us our final reformulation of \eqref{a1} as
\[
(J^n S^*a)_d = (J_+J^n)_d ,
\]
and this follows because the
matrix elements are real and these two operators are adjoints of one another.
\end{proof}
\begin{proof}[Proof of Theorem \ref{TA.1}]
Write $f(z)=\sum_{n\ge 0} f_n z^n$.
The key idea will be to apply the Lax operator $f(J)_a$ to solutions $u$ of $\tau u=zu$.
These sequences $u$ will normally not be in $\ell^2$, but we can (and will) simply interpret
$f(J)_a u$ as the infinite matrix $f(J)_a$ applied to $u$ as a column. Since there are only
finitely many non-zero entries in each row of $f(J)$,
there will be no convergence issues (here we use that $f$ is a polynomial).

Before we do this, we make use of Lemma \ref{LA.1} to rewrite
\begin{align*}
f(J)_a & = \sum f_n \left( 2(J^n)_+ - J^n +(J^n)_d \right) \\
& = 2\sum_{0\le k<n} f_n  a\left( (J^k)_dS-(SJ^k)_d\right) J^{n-k-1} - f(J)+ f(J)_d .
\end{align*}
Now application to a solution $u$ of $\tau u=zu$ yields
\[
f(J)_a u = 2\sum_{0\le k<n} f_n  z^{n-k-1} a\left( (J^k)_dS-(SJ^k)_d\right) u - f(z)u+ f(J)_d u .
\]
Comparison with the expansions \eqref{3.6} of $g,h$, with the coefficients identified with the help of the
Neumann series \eqref{neum}, shows that
\[
f(J)_a u = \left( -2 [fg] aS + [fh]+f(J)_d \right) u ,
\]
and thus
\begin{align*}
X_f(J) u = & \left( -2 [fg] aS + [fh] +f(J)_d \right) zu\\
& + J\left( 2 [fg] aS - [fh] -f(J)_d\right) u  .
\end{align*}
We would now like to write this as a linear combination of two terms: a multiplication operator applied to $u$ and another
one applied to $Su$.
The first three terms are already of this type, and the final three can easily be brought to this form also, if we recall \eqref{jacop}
and use that $(aS+S^*a+b)u=zu$. We omit the details of this straightforward, but tedious calculation and just state the result.
Here, we will use the notations $x_+=Sx$ and $x_-=S^*x$ for shifted sequences, so for example $(x_+)_n=x_{n+1}$;
$Sx$ itself would be ambiguous here, as it could mean the operator (of multiplication) $x$ followed by the operator $S$ or
the operator of multiplication by $(Sx)$. For temporary relief with rapidly growing expressions, we also introduce
the abbreviation $H=[fh]+f(J)_d$ (so $H=H_n(z;J)$ is a sequence with polynomial dependence on $z$). We obtain that
\begin{gather*}
X_f(J)u = KaSu + Lu ,\\
K = 2(z-b_+)[fg_+] - 2(z-b)[fg] +H_- - H_+ ,\\
L = (z-b) (H-H_-) - 2a^2 [fg_+] + 2a^2_-[fg_-] .
\end{gather*}
In particular, evaluation at an $n\in\Z$ gives that
\[
(X_f(J)u)_n = K_na_n u_{n+1} + L_n u_n ;
\]
on the other hand,
\begin{align*}
(X_f(J)u)_n & = (\dot{J}u)_n = \dot{a}_nu_{n+1} + \dot{b}_nu_n + \dot{a}_{n-1}u_{n-1} \\
& = \left( \dot{a}_n - a_n \frac{\dot{a}_{n-1}}{a_{n-1}}\right) u_{n+1} +
\left( \dot{b}_n + (z-b_n)\frac{\dot{a}_{n-1}}{a_{n-1}} \right) u_n .
\end{align*}
Now any two consecutive values of a solution $u$ may be prescribed arbitrarily, so it follows that
\begin{align}
\nonumber
\frac{\dot{a}_n}{a_n} - \frac{\dot{a}_{n-1}}{a_{n-1}} & = K_n, \\
\label{a9}
\dot{b}_n + (z-b_n)\frac{\dot{a}_{n-1}}{a_{n-1}} & = L_n .
\end{align}
The first equation, viewed as a difference equation for $\dot{a}_n/a_n$, has the general solution
\begin{equation}
\label{a5}
\frac{\dot{a}_n}{a_n} = 2(z-b_{n+1})[fg_{n+1}] -H_{n+1}-H_n+C(z;J) ,
\end{equation}
This simplifies considerably because $g,h$ are related by the recursion \cite[Formula (2.188)]{Teschl}
\[
h_{n+1}+h_n = 2(z-b_{n+1}) g_{n+1} .
\]
We obtain that
\begin{equation}
\label{a8}
\frac{\dot{a}_n}{a_n} = -2\,\textrm{\rm Res}(fg_{n+1}) -f(J)_{d,n+1} -f(J)_{d,n} + C ,
\end{equation}
and here $\textrm{Res}(L)$ denotes the (formal) residue of the Laurent series $L$, that is, the coefficient of $z^{-1}$.
We now also see that $C$ must be independent of $z$ because everything else in \eqref{a8} is.
In fact, from the Neumann series \eqref{neum}, it follows that $\textrm{Res}(fg)=-f(J)_d$, so we may rewrite
\eqref{a8} as
\[
\frac{\dot{a}_n}{a_n} = f(J)_{d,n+1}-f(J)_{d,n} + C .
\]
We can now see that $C=0$, as follows: Since $\dot{a}_n = \langle \delta_n, X_f(J) \delta_{n+1} \rangle$,
we deduce that $C$ only depends on those coefficients of $J$ not too far from $n$ (how far exactly is determined
by the degree of $f$). But this holds for any $n$, so $C$ doesn't depend on $J$ at all, and for a $J$ with constant
coefficients (which is fixed by all Toda flows) we clearly must have $C=0$.

By plugging \eqref{a5}
into \eqref{a9}, and specializing to $n=1$ for convenience, we obtain the pair of formulae
\begin{align*}
\frac{\dot{a}_1}{a_1} & = 2(z-b_2)[fg_2] -H_2-H_1 \\
\dot{b}_1 & = -2(z-b_1)^2[fg_1]+2(z-b_1)H_1 +2a_0^2[fg_0]-2a_1^2[fg_2] .
\end{align*}
Now another tedious but straightforward calculation shows that these imply the zero curvature equation \eqref{zc},
for the $B$ from \eqref{defB}.
\end{proof}

\end{document}